\documentclass[11pt]{amsart}

\usepackage{amsmath}
\usepackage{mathrsfs}
\usepackage{amssymb}
\usepackage{amsthm}
\usepackage{bm}

\usepackage{unicode-math}
\usepackage{caption}
\usepackage{booktabs}
\usepackage{xcolor}
\usepackage{listings}
\usepackage[perpage]{footmisc}
\usepackage{xcolor}
\usepackage{hypdoc}

\definecolor{darkblue}{rgb}{0,0,0.4}
\definecolor{darkred}{rgb}{0.5,0,0}
\definecolor{darkgreen}{rgb}{0,0.4,0}
\hypersetup{
    colorlinks=true,
    linkcolor=darkblue,
    citecolor=darkred,
    filecolor=darkgreen,
    urlcolor=darkblue
}
\usepackage{graphicx}
\usepackage{latexsym,wasysym}
\usepackage{mathtools}
\usepackage{color}
\usepackage{url}
\usepackage{fontspec}
\usepackage{enumitem}

\newtheorem{thm}{Theorem}
\newtheorem{cor}[thm]{Corollary}
\newtheorem{lem}[thm]{Lemma}

\newtheorem{rem}[thm]{Remark}

\newcommand{\abs}[1]{\left\vert#1\right\vert}
\newcommand{\R}{\mathbb R}
\newcommand{\Real}{\mathbb R}

\newcommand{\set}[1]{\left\{#1\right\}}

\begin{document}

\title{On the blow-up of harmonic maps from surfaces to homogeneous manifolds}

\author{Hongcan Qian}
\author{Hao Yin}

\address{Hongcan Qian,  School of Mathematical Sciences,
University of Science and Technology of China, Hefei, China}
\address{Hao Yin,  School of Mathematical Sciences,
University of Science and Technology of China, Hefei, China}
\email{qhc001@mail.ustc.edu.cn}
\email{haoyin@ustc.edu.cn}
\thanks{The research work is supported by NSFC-12431003.}

\begin{abstract}
	We study harmonic map sequences from surfaces to compact homogeneous spaces. For sequences developing a single bubble, we derive refined asymptotic expansions in the neck region and prove new obstruction relations among the leading coefficients. These strengthen earlier results by converting an inequality into an equality. For weakly conformal maps, this yields geometric constraints: in low dimensions the tangent planes of the limit map and bubble must coincide, while in higher dimensions they are isoclinic.
\end{abstract}
\maketitle

\section{Introduction}
Suppose that $(M,g)$ and $(N,h)$ are two closed Riemannian manifolds. For a map $u$ from $M$ to $N$, the Dirichlet energy is defined to be
\[
	E(u)=\frac{1}{2}\int_M \abs{\nabla u}^2 dV_g.
\]
The critical points of $E$ are called harmonic maps. When the dimension of $M$ is $2$, this functional $E$ is conformally invariant and the theory of harmonic maps becomes particularly special. Most notably, there is the well-known energy concentration phenomenon for a sequence of harmonic maps with uniformly bounded energy \cite{1}. More precisely, for a sequence of harmonic maps $u_i$ from $M$ (of dimension two) to $N$ with uniformly bounded energy, $u_i$ converges smoothly away from finitely many energy concentration points. Near these points, rescaled sequences of $u_i$ converge to harmonic maps on $\Real^2$, known as bubbles. The annulus regions not captured by these limits are called neck regions. Extensive research has been conducted on the behavior of $u_i$ in the neck region\cite{2,4,9}.

This paper continues in that vein. Before stating our result we need to briefly review the history of this topic. 

For simplicity, we restrict ourselves to the case that $M=B_1$ is the unit ball in $\R^2$ with flat metric and assume that there is one and only one bubble. The general case can be reduced to this one bubble case by a routine induction (see \cite{2}). More precisely, let $u_i$ be a sequence of harmonic maps from $B_1$ to $N$ (isometrically embedded in $\R^l$) with uniformly bounded energy, satisfying

\begin{enumerate}[label=(\roman*)]
\item
 $u_i$ converges locally smoothly in $\bar{B}_1 \textbackslash \set{0}$ to a limit map $u_\infty$, which by the removable singularity theorem from Sacks and Uhlenbeck \cite{1}, can be extended to be a smooth harmonic map from $B_1$ to $N$;

\item

there is a sequence $\lambda_i$ going to zero such that $u_i(\lambda_i y)$ converges locally smoothly to a bubble map $\omega:\R^2 \to N$;

\item

 $\omega$ is the only bubble.
\end{enumerate}

The first nontrivial result is the so called energy identity (see \cite{2,3,4,5}):
\[
\lim_{\delta\to0}\lim_{i\to\infty}\int_{B_\delta\setminus B_{\delta^{-1}\lambda_i}} \abs{\nabla u_i}^2 dx =0,
\]
which means that all the energy of the sequence is accounted for in the limit.
A stronger result known as the no-neck theorem (see \cite{6,7,8}) implies that
\begin{equation}
	\label{eqn:noneck}
	\lim_{\delta\to0}\lim_{i\to\infty}\operatorname{osc}_{B_\delta\setminus B_{\delta^{-1}\lambda_i}}u_i=0,
\end{equation}
which means that the image of the bubble on one end of the neck touches the image of the bubble (or the weak limit) on the other end. The proofs of the above two results amount to showing some decay of the gradient of $u_i$ in the neck region.

In 2022, Yin \cite{yin} proved the following theorem, which gives an expansion of $u_i$ in the neck region. 
\begin{thm}
	\label{thm:yin}
	Let $\{u_i\}$ be a sequence of harmonic maps from $B_1$ to $N$ with uniformly bounded energy, satisfying conditions (i)-(iii) as above. There exist uniformly bounded (vector) coefficients $p_i, q_i, a_{i}, b_{i}, c_{i}$ and $d_{i}$ for any $\alpha\in (0,1)$, such that $u_i$ regarded as function of the cylinder coordinate has the following expansion on $[\log \lambda_i/\delta, \log \delta] \times S^1$,
\begin{equation}
\label{eqn:espansion}
\begin{split}
	u_i(t,\theta)& = p_i + q_i (t-\frac{1}{2}\log \lambda_i) + a_{i} e^t\cos \theta + b_{i} e^t\sin \theta \\
		     &+ c_{i} \lambda_i e^{-t}\cos \theta + d_{i} \lambda_i e^{-t}\sin \theta + O(\eta^{1+\alpha})
\end{split}
\end{equation}
    where $\eta=e^t+\lambda_ie^{-t}$ and $O(\eta^{1+\alpha})$ is some function $w$ satisfying 
    $$
    	\abs{\partial_t^k \partial_\theta^l w}\leq C_{k,l} \eta^{1+\alpha}.
    $$
    Moreover,
\begin{equation}
    \label{eqn:cor_poho}
    \abs{\abs{q_i}^2 -2\lambda_i (a_{i}\cdot c_{i} + b_{i}\cdot d_{i})} \leq C \lambda_i^{\alpha/2+1},\abs{a_i\cdot d_i - b_i\cdot c_i}\leq C \lambda_i^{\alpha/2}.
\end{equation}
\end{thm}
\begin{rem}
Equation \eqref{eqn:espansion} differs from Theorem 1.1 of \cite{yin} in that we have replaced $q_i t$ by $q_i(t-\frac{1}{2}\log \lambda_i)$. Since $\abs{q_i}\leq C \lambda^{1/2}_i$, the difference is an infinitesimal quantity that can be absorbed into $p_i$ without affecting the limit of $p_i$.
\end{rem}

The coefficients in the above expansions for $u_i$ are naturally related to the expansion of the limit map $u_\infty$ at the origin and the bubble map $\omega$ at the infinity. More precisely, by taking subsequence if necessary, $p_i$ converges to $p_\infty$. Moreover,
$$
	\lim_{i\to \infty} a_i =a_\infty, \quad \lim_{i\to\infty} b_i=b_\infty,
$$
where $a_\infty$ and $b_\infty$ appear in the Taylor expansion of $u_\infty$ at the origin
\[
u_\infty= p_\infty+a_\infty x_1+b_\infty x_2 + o(\abs{x}).
\]
Similarly,
$$
	\lim_{i\to \infty} c_i =c_\infty, \quad \lim_{i\to\infty} d_i=d_\infty,
$$
where $c_\infty$ and $d_\infty$ appear in the expansion of the inverted bubble map $\tilde{\omega}$, defined by $\tilde{\omega}(y_1,y_2)= \omega(\frac{y_1}{\abs{y}^2},\frac{y_2}{\abs{y}^2})$, at the origin:
\[
\tilde{\omega}= p_\infty + c_\infty y_1 + d_\infty y_2 + o(\abs{y}).
\]
Taking the limit $i\to \infty$ in \eqref{eqn:cor_poho}, we obtain the following corollary:
\begin{cor}[Corollary 1.2 of \cite{yin}]
	\label{cor:notgood}
\begin{equation}
    \label{eqn:notgood}
a_\infty\cdot c_\infty + b_\infty \cdot d_\infty \geq 0; \quad a_\infty \cdot d_\infty - b_\infty\cdot c_\infty =0.
\end{equation}
\end{cor}
Note that if $a_\infty, b_\infty $ and $c_\infty,d_\infty$ are linearly independent respectively, then the tangent plane of the image of $u_\infty$ at the origin is spanned by $a_\infty$ and $b_\infty$, and the tangent plane of the bubble $\omega$ at the infinity is spanned by $c_\infty$ and $d_\infty$. Then the above restrictions \eqref{eqn:notgood} poses restrictions on how these two tangent planes meet.

While \eqref{eqn:notgood} is not enough to fix the position of the two planes in general, it does have strong consequences when $\dim N=3$ and the maps $u_i$ are weakly conformal. 

The restrictions in \eqref{eqn:notgood} could be understood in two directions. On one hand, it shows that certain blow-up configuration will not happen. On the other hand, it implies that an attempt to construct new harmonic maps by using the gluing method will not be successful unless \eqref{eqn:notgood} is satisfied. Therefore, we would like to call \eqref{eqn:notgood} an obstruction. In \cite{yin}, such obstructions were proved using both the expansion of $u_i$ and the Pohozaev identity.

In this paper, we seek to identify further obstructions of this type when the target manifold is homogeneous. This is a natural direction because, as is well known in the study of harmonic maps (and in variational methods more generally), symmetries lead to conservation laws. The two Pohozaev identities from which \eqref{eqn:cor_poho} is derived are related to the scaling and rotational invariance of the Dirichlet energy. 

As a compact homogeneous space, we may assume that the embedding of $N$ in $\Real^l$ is equivariant. More precisely, if $G$ is the isometry group of $N$, then $G$ is naturally a subgroup of $SO(l)$ and the Lie algebra $\mathfrak g$ is a Lie-subalgebra of $\mathfrak s\mathfrak o(l)$, i.e. skew-symmetric metrics.

\begin{thm}
	\label{thm:main}
	Suppose $N$ is a compact homogeneous space and $\{u_i\}$ is a sequence of harmonic maps satisfying (i)-(iii). For $a_\infty,b_\infty,c_\infty,d_\infty$ defined above, we have
	\begin{eqnarray}
		\label{eqn:abcd}
    a_\infty\cdot c_\infty + b_\infty \cdot d_\infty &=& 0\\
    \label{eqn:adbc}
    a_\infty\cdot d_\infty - b_\infty\cdot c_\infty &=&0.
	\end{eqnarray}
	Moreover, $\abs{q_i}\leq C\lambda_i$. 
\end{thm}
The theorem improves our previous understanding (Corollary \ref{cor:notgood}) in that under the additional assumptions of homogeneous $N$, we are now able to strengthen the first inequality to an equality. 

The following corollary explains the geometric meaning of \eqref{eqn:abcd} and \eqref{eqn:adbc} when $u_i$'s are weakly conformal. Recall that two planes $\pi_1, \pi_2$ in $\R^l$ are called {\bf isoclinic} if for any nonzero $v\in \pi_1$, the angle between $v$ and $\pi_2$ is independent of $v$.

\begin{cor}
	\label{cor:main}
	Let $u_i$ be as in Theorem \ref{thm:main}. We further assume that $u_i$'s are weakly conformal and that none of $a_\infty, b_\infty, c_\infty$ and $d_\infty$ is zero. Let $\pi_1$ be the oriented plane spanned by $a_\infty,b_\infty$ and $\pi_2$ be the oriented plane spanned by $c_\infty,d_\infty$. Then 
\begin{itemize}
	\item If $\dim N=2$ or $3$, then these two planes are the same plane with opposite orientations.
\item If $\dim N\ge 4$, these two planes are isoclinic. \end{itemize}
\end{cor}

\section{Proof of the main theorem}

By \eqref{eqn:cor_poho} in Theorem \ref{thm:yin}, we have
$$
\abs{q_i}\leq C \sqrt{\lambda_i}.
$$
The proof of \eqref{eqn:abcd} and \eqref{eqn:adbc} in Theorem \ref{thm:main} follows from \eqref{eqn:cor_poho} provided we can show that $q_i$ is $o(\lambda_i^{1/2})$. In fact, we will prove a stronger estimate
\begin{equation}
	\label{eqn:qi}
	\abs{q_i}\leq C\lambda_i.
\end{equation}

The proof of \eqref{eqn:qi} consists of two parts. In the first part, we estimate the normal part of $q_i$ (with respect to the tangent space of $N$ at $p_\infty$). To show that the normal part of $q_i$ is bounded by $C\lambda_i$, we need an expansion of $u_i$ up to order $2$ in the neck (see Lemma \ref{lem:higher}).
For the second part, we need first to establish a conservation law for harmonic maps into homogeneous manifold (see Lemma \ref{lem:law}). Then we plug the new expansion into the conservation law (see \eqref{eqn:new}) to get the desired estimate for the tangent part of $q_i$.

{\bf Notation. }For simplicity, we omit the subscript $i$ most of the time in this section.

\subsection{The normal part}

Let $u$ be as in Theorem \ref{thm:main}. By Theorem \ref{thm:yin}, for any $\epsilon>0$, we have the expansion
\begin{equation}
	\label{eqn:newex}	
	\begin{split}
		u(t,\theta)=& p + {q(t-\frac{1}{2}\log \lambda)}+  ae^t \cos \theta + b e^t\sin \theta  \\
			    & + c \lambda e^{-t} \cos\theta + d\lambda e^{-t}\sin \theta + O(2-\epsilon).
	\end{split}
\end{equation}
Here $O(r)$ is short for $O(\eta^r)$ as defined in Theorem \ref{thm:yin}. 

Theorem \ref{thm:yin} is an expansion of $u_i$ up to error $O(2-\epsilon)$. In Lemma \ref{lem:higher}, we improve the error to $O(2)$. In fact, the proof yields a stronger expansion to the order $3-\epsilon$. However, the current form suffices for the proofs in this paper. The proof of Lemma \ref{lem:higher} uses a key lemma from \cite{yin} which we recall:

\begin{lem}[Lemma 3.1 in \cite{yin}]
	\label{lem:key}
	Suppose that $f$ is defined on $[\log \lambda-\log \delta, \log \delta]$ with $\abs{f}\leq C_1 \eta^\alpha$ for some $0<\alpha\notin \mathbb N$. Then we can find a solution $\triangle v =f$ such that 
	$$
	\abs{v}\leq C_2 \eta^\alpha \qquad \text{on} \quad [\log \lambda -\log\delta, \log\delta] \times S^1.
	$$
	Here $C_2$ depends on $C_1$, but not on $\lambda$.
\end{lem}

\begin{lem}
	\label{lem:higher}
	Let $u$ be as above. Then there exist constant vectors $q, p, a, b, c$ and $d$ such that
	\begin{equation}
		\label{eqn:order2exp}
		\begin{split}
		u(t,\theta)=&p+ q(t-\frac{1}{2}\log \lambda) + ae^t \cos \theta + b e^t\sin \theta  \\
		   & + c \lambda e^{-t} \cos\theta + d\lambda e^{-t}\sin \theta + O(2).
		\end{split}
	\end{equation}
\end{lem}
\begin{rem}
	\label{rem:name}
	One may want to compare the coefficients in \eqref{eqn:newex} and \eqref{eqn:order2exp}. We do not claim that they are the same. Indeed, coefficients in such expansions are only determined up to a small error. Obviously, the set of coefficients in \eqref{eqn:order2exp} is valid for \eqref{eqn:newex} and the reverse may not be true. 
\end{rem}

\begin{proof}
	The first step is to plug the known expansion \eqref{eqn:newex} into the right-hand side of the harmonic map equation. By ``1st order of $\nabla u$", we mean 
	$$
	\nabla\left( q(t-\frac{1}{2}\log \lambda) + ae^t \cos \theta + b e^t\sin \theta + c \lambda e^{-t} \cos\theta + d\lambda e^{-t}\sin \theta \right).
	$$
	Direct computation shows
\begin{eqnarray*}
A(u)(\nabla u, \nabla u)&=&A(p)( \text{1st order of $\nabla u$}, \text{1st order of $\nabla u$}) + O(3-\epsilon)\\
&=&A(p)(q,q)+A(p)(a,a)e^{2t}+A(p)(b,b)e^{2t}\\
&&+A(p)(c,c)\lambda^2e^{-2t}+A(p)(d,d)\lambda^2e^{-2t}+2A(p)(q,a)e^t\cos\theta\\
&&+2A(p)(q,b)e^t\sin\theta-2A(p)(q,c)\lambda e^{-t}\cos\theta-2A(p)(q,d)\lambda e^{-t}\sin \theta\\
&&-2A(p)(a,c)\lambda\cos{2\theta}-2A(p)(a,d)\lambda\sin{2\theta}\\
&&-2A(p)(b,c)\lambda\sin{2\theta}+2A(p)(b,d)\lambda\cos{2\theta}+O(3-\epsilon).
\end{eqnarray*}
Setting
\begin{eqnarray*} 
\tilde u&=&\frac12A(p)(q,q)(t-\frac{1}{2}\log \lambda)^2+\frac14A(p)(a,a)e^{2t}+\frac14A(p)(b,b)e^{2t}\\
&&+\frac14A(p)(c,c)\lambda^2e^{-2t}+\frac14A(p)(d,d)\lambda^2e^{-2t} \\
&& +A(p)(q,a)(t-\frac{1}{2}\log \lambda)e^t\cos\theta+A(p)(q,b)(t-\frac{1}{2}\log \lambda)e^t\sin\theta \\
&&+A(p)(q,c)\lambda (t-\frac{1}{2}\log \lambda)e^{-t}\cos\theta +A(p)(q,d)\lambda (t-\frac{1}{2}\log \lambda)e^{-t}\sin\theta \\
&&+\frac12A(p)(a,c)\lambda\cos{2\theta}+\frac12A(p)(a,d)\lambda\sin{2\theta}\\
&&+\frac12A(p)(b,c)\lambda\sin{2\theta}-\frac12A(p)(b,d)\lambda\cos{2\theta},
\end{eqnarray*} 
we have 
\begin{eqnarray*} 
\triangle\tilde u&=&A(p)(q,q)+A(p)(a,a)e^{2t}+A(p)(b,b)e^{2t}\\
&&+A(p)(c,c)\lambda^2e^{-2t}+A(p)(d,d)\lambda^2e^{-2t}+2A(p)(q,a)e^t\cos\theta\\
&&+2A(p)(q,b)e^t\sin\theta-2A(p)(q,c)e^{-t}\lambda\cos\theta-2A(p)(q,d)e^{-t}\lambda\sin \theta\\
&&-2A(p)(a,c)\lambda\cos{2\theta}-2A(p)(a,d)\lambda\sin{2\theta}\\
&&-2A(p)(b,c)\lambda\sin{2\theta}+2A(p)(b,d)\lambda\cos{2\theta}.
\end{eqnarray*}
Then by the harmonic map equation of $u$, we get
$$
	\triangle (u-\tilde{u}) = O(3-\epsilon).
$$
By Lemma \ref{lem:key}, 
$$
	u = \tilde{u} + h + O(3-\epsilon)
$$
where $h$ is a harmonic function.

Claim: $\tilde{u}=O(2)$. To see this, we note the following facts
\begin{itemize}
	\item $\lambda=O(2)$;
	\item $q = O(1)$;
	\item $q (t-\frac{1}{2}\log \lambda)= O(1)$.
\end{itemize}
Similar to Lemma 4.3 of \cite{yin}, we have bounds on the coefficients of the expansion of $h$ so that we can absorb higher order terms (no less then order $2$) into $O(2)$. This concludes the proof of \eqref{eqn:order2exp}. Note that we use the same set of symbols for the coefficients as in \eqref{eqn:newex} and the reason is explained in Remark \ref{rem:name}.
\end{proof}

Intuitively, we expect $p\in N$. However, this does not follow from the proof. To address this issue, we prove the next lemma.
\begin{lem}
	\label{lem:p}
	Let $u$ be as in Lemma \ref{lem:higher}. Then
	$$
	{\rm dist}(p,N) \leq C\lambda.
	$$
\end{lem}
\begin{proof}
	By the no-neck theorem (see \eqref{eqn:noneck}), the image of $u$ on the cylinder lies in a small neighborhood $U$ of $N$. We can assume that there exist $m=p-\dim N$ functions $\varphi_j:U\to\R$ s.t. $N\cap U= \bigcap_{j=1}^m \varphi_j^{-1}(\set{0})$. 

	Letting $t_0=\frac{1}{2}\log \lambda$, we have
	$$
	\varphi_j(u(t_0,\theta)) = \varphi_j(p) + (\nabla \varphi_j)(p) \left( (a+c) \sqrt{\lambda} \cos \theta + (b+d) \sqrt{\lambda} \sin \theta  \right) + O(2).
	$$
Since $u(t_0,\theta)\in N$, $\varphi_j(u(t_0,\theta))=0$ for all $\theta$. Integrating the above equation over $S^1$, we find
$$
	0 = 2\pi \varphi_j(p) + O(2).
$$
The proof is done.
\end{proof}
As a corollary, we assume without loss of generality that $p$ in \eqref{eqn:order2exp} lies in $N$. 

Now we estimate the normal part of $q$. By the definition of $\varphi_j$ and $p\in N$, the space of normal vectors at $p$ is spanned by $\nabla \varphi_j(p)$ for $j=1,\cdots ,m$.
\begin{lem}\label{lem:smalltangent}
	Let $u$ be as in Lemma \ref{lem:higher}. we have
\begin{equation}
	\label{eqn:normalq}
	q\cdot \nabla \varphi_j(p) = O(2)\quad \forall j.
\end{equation}
\end{lem}
\begin{proof}
	Since $u$ maps into $N$, then for all $(t,\theta)$, we have
\begin{equation}
    \label{eqn:stayonN}
	\partial_t u \cdot \nabla \varphi_j(u) =0.
\end{equation}
When $t=t_0=\frac{1}{2}\log \lambda$, we have the expansion of $\nabla \varphi_j$ at $p$ 
\[
\nabla\varphi_j(u)(t_0,\theta)=\nabla\varphi_j(p)+\nabla^2\varphi_j(p)(\sqrt\lambda((a+c)\cos\theta)+(b+d)\sin\theta)+O(2)
\]
and the expansion of $\partial_t u$ at $p$
$$
\partial_t u (t_0,\theta) =q +  a\sqrt\lambda\cos \theta + b\sqrt\lambda\sin \theta - c\sqrt\lambda  \cos\theta - d\sqrt\lambda\sin \theta + O(2). 	
$$
Plugging them into \eqref{eqn:stayonN}, we obtain
\begin{equation*}
0=q\cdot \nabla \varphi_j(p)+(\cdots)\sin\theta+(\cdots)\cos\theta+O(2),
\end{equation*}
where $(\cdots)$ are something doesn't affect the conclusion. By integrating the above equation over $\set{t_0} \times S^1$, we get the conclusion we need.
\end{proof}

\subsection{The tangent part}

To estimate the tangent part, we first prove a result valid for all harmonic maps into homogeneous space. Recall that $\mathfrak{g}$ is the Lie-algebra of the isometry group of $N$ embedded in $\mathfrak{so}(\Real^l)$.
\begin{lem}
	\label{lem:law}
	Let $\Omega$ be a domain in $M$ and $B$ be any matrix in $\mathfrak g$. For any harmonic map $u$ from $M$ to $N$, we have
	\[
		\int_{\partial\Omega} \frac{\partial u}{\partial\mathbf{n}} \cdot Bu =0,
	\]
	where $\mathbf{n}$ is the outward normal vector of $\Omega$ in $M$.
\end{lem}
\begin{proof}
	Let $B$ be any element in $\mathfrak g$.
Since $\exp(tB)$ is an isometry of $\Real^l$ for all $t$, we have
\[
0=\frac{d}{dt}|_{t=0}\int_\Omega\abs{\nabla(\exp(tB) u)}^2=2\int_\Omega\nabla u\cdot\nabla Bu.
\]
Since $B \in \mathfrak{g}$ generates a one-parameter group of isometries $\exp(tB)$ on $N$, the curve $\gamma(t) = \exp(tB)u$ lies on $N$, and $\gamma'(0) = Bu$, we have $Bu \in T_u N$.
By the harmonic map equation of $u$ and the fact that $Bu$ is a tangent vector of $N$ at $u$, we have
$$
\int_\Omega\triangle u\cdot Bu=0.
$$
Integrating by parts, we get
\[
		\int_{\partial\Omega} \frac{\partial u}{\partial\mathbf{n}} \cdot Bu =0.
\]
\end{proof}

Applying the above lemma on the disk, we obtain
\begin{cor}
	Let $t_0=\frac{1}{2}\log \lambda$ and $B$ be any matrix in $\mathfrak g$. For $u$ in Theorem \ref{thm:main}, we have
\begin{eqnarray}
        \label{eqn:new}
		\int_{\set{t_0} \times S^1 } \partial_t u \cdot Bu d\theta =0.
\end{eqnarray}
\end{cor}

Substituting the expansion in Lemma \ref{lem:higher} into \eqref{eqn:new}, we get
\begin{eqnarray*}
0&=&\int_{\set{t_0} \times S^1 } \partial_t u \cdot Bu d\theta \\
&=&\int_{\set{t_0} \times S^1 } (q + \sqrt{\lambda}(a-c)\cos \theta+\sqrt{\lambda}(b-d)\sin \theta + O(2)\\
&&\cdot \left(Bp+\sqrt{\lambda}B(a+c)\cos\theta+\sqrt{\lambda}B(b+d)\sin\theta+O(2)\right)d\theta \\
&=& 2\pi \left( q \cdot Bp + \frac{1}{2}\lambda (a-c)\cdot B(a+c) + \frac{1}{2}\lambda (b-d)\cdot B(b+d)  \right) + O(2).
\end{eqnarray*}
Since $B$ is skew-symmetric, we get
$$
	q\cdot B p + \lambda a\cdot Bc + \lambda b\cdot Bd = O(2).
$$
Since $Bp$ can be any tengent vector at $p$, we find that the tangent part of $q$ is $O(2)$.

Combining the normal part and the tangent part, we find that $\abs{q}=O(2)$, which is equivalent to \eqref{eqn:qi}. As pointed out at the beginning of this section, \eqref{eqn:abcd} and \eqref{eqn:adbc} follows from \eqref{eqn:qi} and the proof for Theorem \ref{thm:main} is done.

\section{Proof of Corollary \ref{cor:main}}
\begin{lem}
    Let $u_i$ be as in Theorem \ref{thm:main}. We further assume that $u_i$'s are weakly conformal. Then
\begin{eqnarray}
 \label{eqn:rel1} \abs {a_\infty}=\abs {b_\infty},\quad \abs {c_\infty}&=&\abs {d_\infty}\\
 \label{eqn:rel2}    a_\infty \cdot b_\infty = c_\infty \cdot d_\infty &=& 0 
\end{eqnarray}
\end{lem}
\begin{proof}
By the assumption that $u_i$ is weakly conformal, we can get the $u_\infty$ is weakly conformal.Consider the expansion at the origin:
\[
u_\infty= p_\infty+a_\infty x_1+b_\infty x_2 + o(\abs{x}).
\]
Since $u_\infty$ is weakly conformal,we have 
\[
\abs{a_\infty}=\abs{b_\infty},a_\infty\cdot b_\infty=0.
\]
Similarly,we can get
\[
\abs{c_\infty}=\abs{d_\infty},c_\infty\cdot d_\infty=0.
\]
\end{proof}

Now we turn to the proof of Corollary \ref{cor:main}.
\begin{proof}[Proof of Corollary \ref{cor:main}]
Denote $\dim N$ by $n$. For simplicity, we omit the subscript $\infty$.
Since $a,b,c,d$ lie in $T_p N$ and \eqref{eqn:abcd} and \eqref{eqn:adbc} are still true when we regard $a,b,c,d$ as vectors in $T_p N$. Therefore, in what follows, we assume that $a,b,c,d$ are all $n$-tuple vectors.

{\bf Case $n=2$:}  By a rotation, we may assume 
\[
a=(1,0), \quad b=(0,1), \quad c=(c_1,c_2), \quad d=(d_1,d_2).
\]
Substituting these in \eqref{eqn:abcd} and \eqref{eqn:adbc}, we can get 
$$c_1=-d_2,\quad c_2=d_1,$$
which means that $\pi_1$ and $\pi_2$ is the same plane with opposite orientations.

{\bf Case $n=3$:}
Again, we assume that
\[
   a=(1,0,0),\quad b=(0,1,0), \quad c=(c_1,c_2,c_3), \quad d=(d_1,d_2,d_3).
\]
By \eqref{eqn:abcd} and \eqref{eqn:adbc}, we get
\[
c_1=-d_2, \quad c_2=d_1,
\]
hence $c_1^2+c_2^2=d_1^2+d_2^2$.

By \eqref{eqn:rel1} , we obtain
$$\abs {c_3}=\abs{d_3},$$
and by \eqref{eqn:rel2}, we obtain
$$c_3\cdot d_3=0,$$
which implies that $c_3=d_3=0$.
Hence, the conclusion is the same as in the case $n=2$.

{\bf Case $n\geq 4$:} There is always a subspace of dimension four containing $\pi_1$ and $\pi_2$. Hence, we may assume $n=4$ and
\[
   a=(1,0,0,0),b=(0,1,0,0),c=(c_1,c_2,c_3,c_4),d=(d_1,d_2,d_3,d_4).
\]
Similar to the previous cases, we still have
\begin{equation}
    \label{eqn:again}
c_1=-d_2, \quad c_2=d_1.
\end{equation}
With the condition \eqref{eqn:rel1} and \eqref{eqn:rel2}, we obtain that
\[
c_3^2+c_4^2=d_3^2+d_4^2, \quad c_3\cdot d_3+c_4\cdot d_4=0.
\]

For any nonzero vector $v\in \pi_2$, there exists $(s,t)\in \R^2$ with $s^2+t^2>0$ such that
\[
v= s c + td.
\]
Its projection onto $\pi_1$ is
\[
\nu=( sc_1+td_1, sc_2+td_2,0,0)
\]
By \eqref{eqn:again}, we have
\[
\nu=(sc_1-tc_2,sc_2+tc_1,0,0).
\]
Due to \eqref{eqn:rel1} and \eqref{eqn:rel2}, the cosine of the angle between $v$ and $\pi_1$ is

\begin{eqnarray*}
\cos\alpha&=&\frac{\abs\nu}{\abs{sc+td}}\\
          &=&\frac{\sqrt{(s^2+t^2)(c_1^2+c_2^2)}}{\sqrt{(s^2+t^2)}\abs{c}} \\
          &=&\frac{\sqrt{(c_1^2+c_2^2)}}{\abs{c}},
\end{eqnarray*}
which is independent of $s,t$. Hence $\pi_1,\pi_2$ are isoclinic.
\end{proof}

\bibliography{foo}
\bibliographystyle{alpha}

\end{document}